\documentclass[11pt]{article}


\newcounter{parentnumber}

\usepackage{amsmath, amssymb, amscd, amsthm, amsfonts}
\usepackage{graphicx}
\usepackage{hyperref}
\usepackage{soul}
\usepackage{color}
\usepackage{tikz}
\usepackage{wrapfig}
\usepackage{subcaption} 

\title{On the illumination of centrally symmetric cap bodies in small dimensions}
\author{Ilya Ivanov \hfill and \hfill Cameron Strachan}
\date{}

\newtheorem{definition}{Definition}
\newtheorem{rmk}{Remark}
\newtheorem{theorem}{Theorem}
\newtheorem{lemma}[theorem]{Lemma}
\newtheorem{conjecture}[theorem]{Conjecture}

\newcommand{\norm}[1]{\left\lVert#1\right\rVert}
\newcommand{\Ed}{\mathbb{E}^d}
\newcommand{\Ee}{\mathbb{E}}
\newcommand{\Ss}{\mathbb{S}^{d-1}}
\newcommand{\Sp}{\mathbb{S}}
\newcommand{\St}{\mathbb{S}^3}
\newcommand{\ud}{\boldsymbol{u}}
\newcommand{\pd}{\boldsymbol{p}}
\newcommand{\qd}{\boldsymbol{q}}
\newcommand{\xd}{\boldsymbol{x}}
\newcommand{\ed}{\boldsymbol{e}}
\newcommand{\vd}{\boldsymbol{v}}
\newcommand{\Od}{\boldsymbol{o}}
\newcommand{\rr}{\mathbb{R}}
\newcommand{\orig}{\mathbf{o}}
\newcommand{\lacket}{\left\{}
\newcommand{\racket}{\right\}}

\DeclareMathOperator{\conv}{conv}

\DeclareMathOperator{\bd}{bd}
\DeclareMathOperator{\inter}{int}

\DeclareMathOperator{\vol}{vol}

\DeclareMathOperator{\Hem}{Hem}

\begin{document}

\setcounter{page}{1}

\maketitle

\begin{abstract}
The illumination number $I(K)$ of a convex body $K$ in Euclidean space $\mathbb{E}^d$ is the smallest number of directions that completely illuminate the boundary of a convex body. A cap body $K_c$ of a ball is the convex hull of a Euclidean ball and a countable set of points outside the ball under the condition that each segment connecting two of these points intersects the ball. The main results of this paper are the sharp estimates $I(K_c)\leq6$ for centrally symmetric cap bodies of a ball in $\mathbb{E}^3$, and $I(K_c)\leq 8$ for unconditionally symmetric cap bodies of a ball in $\mathbb{E}^4$.
\end{abstract}

\textbf{Mathematics Subject Classification:} 52A20, 52A55.\\

\textbf{Keywords:} cap body, illumination number, separation by hemispheres

\section{Introduction}\label{section-introduction}

\subsection{On the status of the Illumination Conjecture}

Let $\Ed$ denote a $d$-dimensional Euclidean space. The origin point is denoted $\orig$, and $\Ss$ is an origin-centered $(d-1)$-sphere with a unit radius. We say that $K \subset \Ed$ is a \textit{convex body} if it is a compact, convex subset of $\Ed$ with a non-empty interior. Unless specified otherwise, it is implied that the convex body contains the origin $\orig$ in its interior.

    Consider a convex body $K$, a point $\pd$ on its boundary, and a point $\ud$ on $\Ss$. We say $\ud$ \textit{illuminates} $\pd$ if and only if there exists a point in the interior of $K$,  $\qd \in \inter K$, such that $\qd = \pd + \lambda \ud$ for some positive $\lambda \in \mathbb{R}^+$. In other words, the ray with direction $\ud$ that starts at $\pd$  has to pierce the interior of $K$. $K$ is \textit{completely illuminated} by a set of directions $U = \left\{\ud_1, \ud_2, \dots, \ud_k \right\} \subset \mathbb{S}^{d-1}$ if every point in $\bd K$ is illuminated by at least one direction from $U$. The illumination number of $K$, $I(K)$, is the smallest number of directions that completely illuminate $K$.

\begin{conjecture}[\textbf{Illumination Conjecture}]\label{illconj}
The illumination number $I(K)$ of any d-dimensional convex body $K \subset \Ed$ does not exceed $2^d$. Moreover, $I(K) = 2^d$ if and only if $K$ is an affine image of a d-cube.
\end{conjecture}

The illumination conjecture has several alternative statements. It was first posed in 1955 as a covering problem in two dimensions.  Levi proved that any 2-dimensional convex body can be covered with 4 translates of its interior \cite{leviUeberdeckungEibereichesDurch1955}. Later, Hadwiger \cite{hadwigerUngelosteProblemeNr1957} conjectured that any d-dimensional convex body can be covered by $2^d$ translates of its interior. Independently, Markus and Gohberg \cite{gohbergCertainProblemCovering1960} have posed a similar conjecture about covering a $d$-dimensional convex body $K$ with translates of its smaller homothetic copies $\ud+\lambda K,$ where $\ud \in \Ed,$ and $ \lambda \in (0,1)$. 

The question of illumination by directions was first stated by Boltyanskii \cite{boltyanskiProblemIlluminatingBoundary1960}. Hadwiger \cite{hadwigerUngelosteProblemeNr1960} has offered a point source interpretation of the problem. Instead of illuminating a region of $\bd K$ by all the rays with the same direction, it is illuminated by all the rays starting in a point outside $K$. 

For a convex body $K \subset \Ed$ two covering numbers can be defined: $C_{hom}(K)$, the smallest number of smaller homothets of $K$ required to cover it, and $C_i(K)$, the smallest number of interior translates required to cover $K$. There are also two illumination numbers: $I_p(K)$, the smallest number of point sources that would completely illuminate $K$, and, finally, the direction illumination number $I(K)$ defined above. For any convex body $K \subset  \Ed$ all these numbers are equal (for details see \cite{boltyanskiExcursionsCombinatorialGeometry1996}). 
\begin{equation}
    C_i(K) = C_{hom}(K) = I_p(K) = I(K)
\end{equation}

So far, the illumination conjecture has been completely proven only in $\mathbb{E}^2$. Best general estimates for the convex body illumination numbers in $\Ee^3, \Ee^4, \Ee^5, \Ee^6$ are, respectively, $I(K)\leq 16$ \cite{papadoperakisEstimateProblemIllumination1999}, $I(K) \leq 96, I(K) \leq 1091,$ and $I(K) \leq 15373$ \cite{prymak_illumination_2020}. In $\Ee^3$ the illumination conjecture is proven for the convex bodies with central symmetry \cite{lassakSolutionHadwigerCovering1984}, bodies with symmetry about a plane \cite{deksterEachConvexBody2000}, and polytopes with affine symmetry \cite{bez_affine}. 

For a long time, the best general estimate was due to Rogers' work \cite{rogersNoteCoverings1957} and his collaborations with Shepard \cite{rogersDifferenceBodyConvex1957} and Erdos \cite{erdosCoveringSpaceConvex1962}:
\begin{equation}
    I(K) \leq \frac{\vol_{d} (K - K)}{\vol_{d}(K)} d(\ln d+\ln\ln d+5)\leq {\binom{2d}{d}}d(\ln d+\ln\ln d+5)=O(4^{d}\sqrt{d}\ln d).
\end{equation}
Which for centrally symmetric convex bodies turns into:
 \begin{equation}
    I(K) \leq \frac{\vol_{d} (K - K)}{\vol_{d}(K)}d(\ln d+\ln\ln d+5)=2^{d}d (\ln d+\ln\ln d+5)= O(2^{d}d\ln d). 
\end{equation}
Recently the general Rogers' estimate was improved in the paper \cite{huangImprovedBoundsHadwiger2018} to $I(K) \leq c_1 4^d e^{-c_2 \sqrt n}$ for some  universal constants $c_1$ and $c_2$ .
 
For a more detailed outlook of the illumination conjecture see \cite{bezdekGeometryHomotheticCovering2018}. 

\begin{definition}
A \textbf{cap body of a ball}  is the convex hull of the closed ball  $B^d[\orig,r] \subset \mathbb{E}^d$ and a countable set of points outside the ball $\left\{\vd_i \in \Ed \setminus B^d[\orig,r] | i \in I \right\}$
such that for any pair of distinct points $\vd_i, \vd_j$ with $ i,j \in I$, the line segment $\overline{\vd_i\vd_j}$ intersects the closed ball.
\end{definition}

 Since the illumination number is invariant with respect to affine transformations, we will only consider the cap bodies based on the origin-centered ball with a unit radius. Here and after, unless specified otherwise, we will use ``cap body'' to refer to the cap body of an origin-centered unit ball. See Fig. \ref{fig:2dcap} for an example of 2-dimensional cap body.

\begin{figure}[h]
\centering
\includegraphics[scale=0.4]{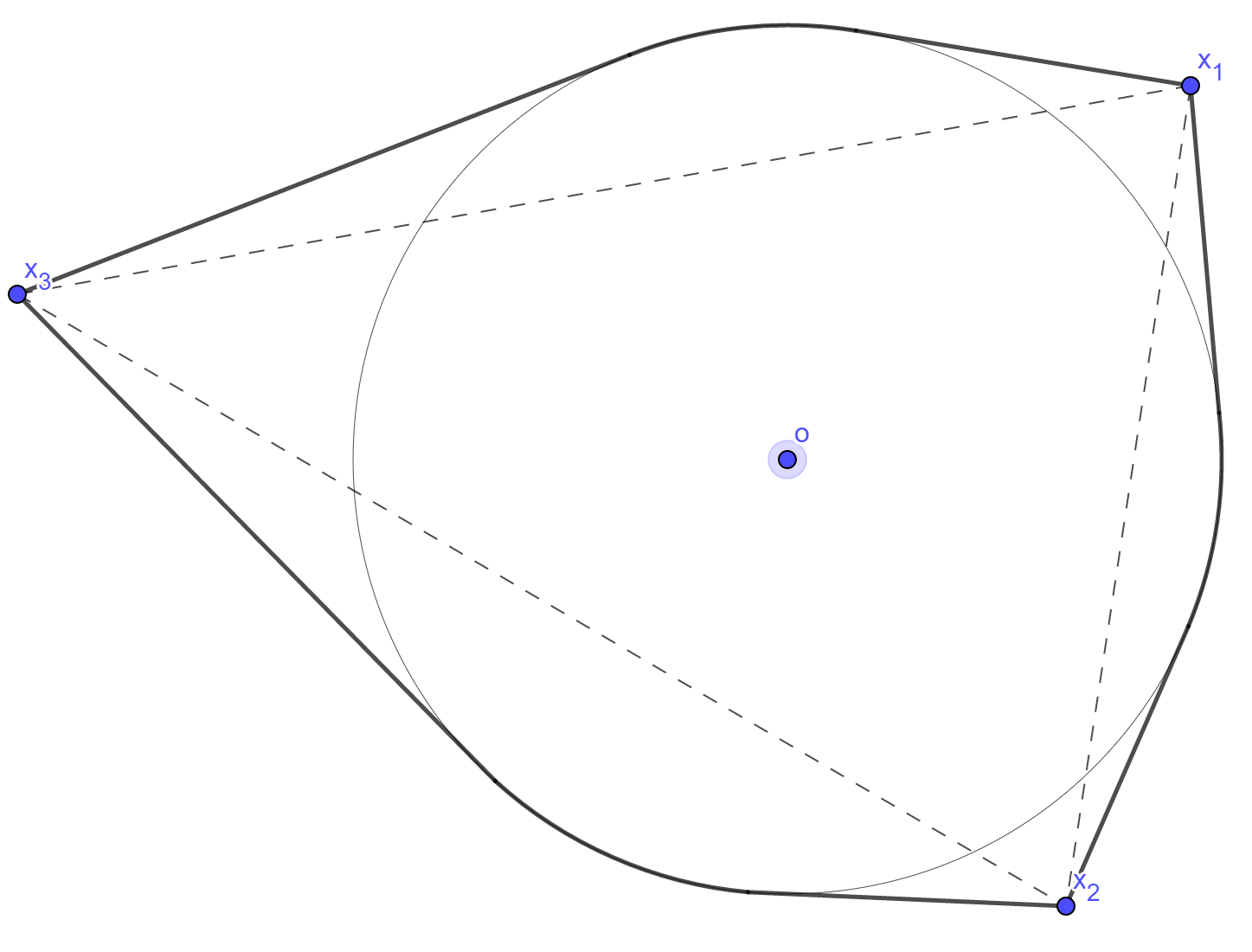}
\caption{A cap body in $\Ee^2$ with three vertices}
\label{fig:2dcap}
\end{figure}

 Cap bodies of a ball were first introduced by Minkowski in 1903 \cite{minkowskiVolumenUndOberflache1903}. Minkowski has conjectured that only the cap bodies maximize the product of volume and mean width given a fixed surface area. This conjecture was proven later \cite{bolBeweisVermutungMinkowski1943}. For the details and recent results on Minkowski's quadratic inequality extremals see \cite{shenfeldExtremalsMinkowskiQuadratic2019}.

Cap bodies were used by Naszodi in \cite{nasspiky}, under the name of a ``spiky ball'' in the illumination context. Naszodi used the construction to demonstrate that for any positive $\varepsilon$ there is a $d$-dimensional cap body in a $\varepsilon$-region of a Euclidean ball with an illumination number exponentially large with respect to $d$.

\subsection{New results}

Let $\left\{ \ed_i \mid i \in \{1 \dots d \right\}\}$ be the standard orthonormal basis. For a point $\ud \in \Ed$ and $\alpha \in \rr$ we will denote by $H_{\ud,\alpha}$ the hyperplane given by the equation $\left\{ \xd \in \Ed \mid \langle \xd, \ud \rangle = \alpha \right\}$. For that hyperplane,  $H^+_{\ud, \alpha} = \lacket \xd \in \Ed \mid \langle \xd, \ud \rangle \geq \alpha \racket$ is its positive halfspace and $H^-_{\ud, \alpha}=\lacket \xd \in \Ed \mid \langle \xd, \ud \rangle \leq \alpha \racket$ is its negative halfspace. The hyperplanes $H_{\ed_i, 0}$ we will call \textit{coordinate hyperplanes}. Respectively, $(d-2)$-greatspheres $G_i = \Ss \cap H_{\ed_i, 0}$ are \textit{coordinate greatspheres}. If $\ud$ is a unit vector, $\Hem_{\ud}$ stands for the open hemisphere of $\Ss$ with centre in $\ud$.

There is a correspondence between vertices $\vd_i$ of a cap body $K_c = \conv \left(  \Ss \cup  \left\{  \vd_i \mid i \in I \right\} \right)$ and spherical caps on $\Ss$, 
$C_i = \Ss \cap H^+_{\vd_i,1} \mid  i \in I$. These caps form a packing on $\Ss$, their interiors do not intersect.
We will prove that the problem of illuminating a cap body by a set of directions is equivalent to the problem of covering spherical caps with open hemispheres. If the spherical cap $C$ is a subset of an open hemisphere $\Hem_{\ud}$ for some $\ud \in \Ss$, we say that $\Hem_{\ud}$  \textit{separates} $C$. Similarly, hyperplane $H$ passing through the origin separates the cap $C$ if one of the open hemispheres $\Ss \setminus H^{\pm}$ separates the cap. In this case we also say that the $(d-2)$-greatsphere $H \cap \Ss$ separates the cap.

\begin{theorem}[\textbf{Cap Body Illumination Criterion}]\label{thm:cbcriterion}
A cap body $K_c = \conv \left( \Ss \cup \left\{\vd_i \mid i \in I\right\} \right)$ is illuminated by directions $\ud_1, \dots, \ud_k \in \Ss$ if and only if each closed spherical cap $\Ss \cap H^+_{\vd_i,1},$ for $i \in I$ is separated by some open hemisphere from $\Hem_{-\ud_j}$, $j \in \{1,\dots,k\}$, and this set of hemispheres completely covers $\Ss$.
\end{theorem}

\begin{definition}
An \textbf{unconditionally symmetric cap body} in $\Ed$ is a cap body which is symmetric about every coordinate hyperplane $H_{\ed_i,0}$ for $i \in \{1,2, \dots, d\}$.
\end{definition}

We study the illumination of centrally symmetric cap bodies in $\mathbb{E}^3$ and unconditionally symmetric cap bodies in $\mathbb{E}^4$. The main results of this paper are the theorems \ref{thm:S2} and \ref{thm:S3}: 

\begin{theorem}\label{thm:S2}
The illumination number of a centrally symmetric cap body of a ball in $\Ee^3$ is at most 6, and this estimate is sharp.
\end{theorem}

\begin{theorem}\label{thm:S3}
The illumination number of an unconditionally symmetric cap body of a ball in $\Ee^4$ is at most 8, and this estimate is sharp.
\end{theorem}

\begin{rmk}
The illumination conjecture has already been proven for centrally symmetric convex bodies in $\Ee^3$ \cite{lassakSolutionHadwigerCovering1984}. We sharpen the illumination number estimate for the centrally symmetric cap bodies to 6, compared to Lassak's general estimate of 8. Similarly, in $\Ee^4$ we show that the illumination number of an unconditionally symmetric cap body is at most 8, compared to $2^4$ that features in the general illumination conjecture statement.
\end{rmk}

In section \ref{sec:s2proof}  we show that for the centrally symmetric cap bodies in $\mathbb{E}^3$, there always exist three pairwise orthogonal great circles that separate every cap on the sphere. We pick a cap $C_{max}$ of a largest radius and position the cap body so that $\ed_3$ is the centre of this cap. Caps that are not separated by the coordinate great circle $G_3$ are the ones that intersect it. We then show that the great circles $G_1, G_2$ can be rotated around $\ed_3$ so that all these caps are separated.

For the cap bodies in $\Sp^3$ we consider only the cap bodies with caps that are not separated by the 4 coordinate greatspheres. We show in section \ref{sec:ktan} that due to the unconditional symmetry,  the caps that fail to be separated by this configuration have to be tangent to $k$ coordinate greatspheres ($1 < k \leq d$), and have their centers on the remaining $d-k$ coordinate greatspheres. We call such caps $k$-tangent.

Since the caps are packed on the sphere, only four distinct configurations of $k$-tangent caps are allowed (up to orthogonal transformation
). We consider all the possible configurations of $k$-tangent caps, for each configuration we pick four pairwise orthogonal 2-greatspheres that separate all the $k$-tangent caps, and then we show that these greatspheres would also separate any other cap that forms a packing with the $k$-tangent caps.

\begin{rmk}
The illumination number of the convex body is invariant with respect to affine transformations. Therefore, our results are also applicable to the cap bodies of ellipsoids, affine images of the cap bodies of the balls.
\end{rmk}

\section{Proof of Theorem \ref{thm:cbcriterion}}

 Let $K_c = \conv \left( \Ss \cup \left\{\vd_i \mid i \in I \right\} \right)$ denote a cap body in $\Ed$.

\begin{lemma}
\label{illum_innp} A vertex $\vd_i$ of a cap body $K_c$ is illuminated by the direction $\ud \in \Ss$ if and only if $\langle \vd_i, \ud \rangle < -\sqrt{\norm{\vd_i}^2-1}$

\end{lemma}

\begin{proof}
Suppose direction $\ud$ illuminates the vertex $\vd_i$. This takes place if and only if the ray starting at the point $\vd_i$ with the direction $\ud$ intersects the plane $H_{\vd_i,1}$ in a point inside the $\Ss$. In other words, there is a positive $\lambda$ such that 
$\langle \vd_i, \vd_i+\lambda \ud \rangle = 1 \left( \mbox{which is equivalent to }\lambda = \frac{1-\norm{\vd_i}^2}{\langle \ud, \vd_i \rangle}\right)$ and $\norm{\vd_i + \lambda \ud} < 1$. Combining these two conditions concludes the proof of the lemma.
\end{proof}

    \begin{lemma}
    A boundary point of a cap body $\pd \in \bd K_c$ that is also on the sphere, $\pd \in \Ss$, is illuminated by the direction $\ud \in \Ss$ if and only if $\langle \pd,\ud \rangle <0$
    \end{lemma}
    
    \begin{proof}

   The point $\pd$ is illuminated by $\ud$ if and only if there is a non-negative  $\lambda$ such that $\norm{\pd + \lambda \ud} < 1$, or, equivalently, $\norm{\pd + \lambda \ud}^2 < 1$. Using the fact that $\norm{\pd} = \norm{\ud} = 1$ yields $2 \lambda \langle \pd,\ud \rangle <-1$. It holds for some  $\lambda \geq 0$ if and only if  $\langle \pd,\ud \rangle<0$.

    \end{proof}
     \begin{definition}
   A \textbf{spike} of a vertex $\vd_i, i \in I$ of a cap body $K_c$ is the set $S_i = \bd \conv ( \Ss \cup \vd_i) \setminus \Ss$.
   \end{definition}
   
    Note that any point on the boundary of a cap body either lies on the $\Ss$, or on a spike $S_i$ of some vertex $\vd_i$.
    
    \begin{lemma}
    Let $\vd_i \in \Ed \setminus \Ss$ be a vertex of a cap body $K_c$. Then every point on the spike $S_i$ is illuminated by the direction $\ud \in \Ss$ if and only if the vertex $\vd_i$ is illuminated by the direction $\ud$  
    \end{lemma}
    
    \begin{proof}
    
    The ``only if'' part follows from the fact that $\vd_i \in S_i$. 
    
    Suppose that the direction $\ud \in \Ss$ illuminates the vertex $\vd_i$. Now we need to show that an arbitrary point $\pd \in S_i$ is also illuminated by $\ud$. Let $\qd$ be the point where the line through $\vd_i$ and $\pd$ meets $\Ss$. Then ``shrink'' the spike: let $S'_i$ be the homothet of $S_i$ with centre $\qd$ and the homothety coefficient $\frac{\norm{\pd - \qd}}{\norm{\vd_i - \qd}}$, so that $\pd$ is the the image of $\vd_i$ under this transformation. Since for some $\varepsilon\in \mathbb{R}^+$ there is a point $\vd_i + \varepsilon \ud \in \inter \conv S_i$, there is also a point $\pd + \frac{\norm{\pd - \qd}}{\norm{\vd_i - \qd}} \varepsilon \ud \in \inter \conv S'_i \subset \inter \conv S_i$, and hence, $\pd$ is illuminated by $\ud$.
    
    \end{proof}

    \begin{lemma}
    \label{VertCap}
    Let $\vd_i \in \Ed \setminus \Ss$ be a vertex of a cap body $K_c$. Then $\vd_i$ is illuminated by the direction $\ud \in \Ss$ if and only if the closed spherical cap $C_i = \left\{ \pd \in \Ss \mid \langle \pd,\vd_i \rangle \geq 1 \ \right\}$ is separated by $Hem_{-\ud}$.
    \end{lemma}
    
    \begin{proof}
    
    Cap $C_i$ lies in the hemisphere $Hem_{-\ud}$ if and only if the angle between $\ud$ and the centre of $C_i$ (which is equal to the angle between $\ud$ and $\vd_i$) is greater than $\pi/2 + r_i$, where $r_i$ is the spherical radius of the cap $C_i$. This condition is equivalent to  $\frac{\langle \ud,\vd_i \rangle}{\norm{\vd}} \leq \cos \left( \pi/2 + r_i \right)$. This can be transformed into $\langle \vd_i, \ud \rangle < -\sqrt{\norm{\vd_i}^2-1}$ using the fact that $\cos r_i = 1/\norm{\vd_i}$. Together with Lemma \ref{illum_innp} this concludes the proof.

    \end{proof}
    
We have shown that every spike of a cap body is illuminated by a set of directions if and only if every corresponding spherical cap of a cap body is separated by the corresponding set of hemispheres. That concludes the proof of the Theorem \ref{thm:cbcriterion}. 

\section{Proof of the Theorem \ref{thm:S2}}\label{sec:s2proof}

Let $K_c \subset \mathbb{E}^3$ be a cap body symmetric with respect to the origin, with vertices $\left\{\vd_i \mid i \in I\right\}$. We want to prove that there exist six illumination directions $\ud_1, \dots, \ud_6$ such that every closed cap $C_i = \Sp^2 \cap H^+_{\vd_i,1}$ along with every other point on $\mathbb{S}^2$ belongs to at least one open hemisphere $Hem_{-\ud_j}$ . 

\begin{definition}
A \textbf{view angle} of a spherical cap $C \subset \Sp^2$ from the point $\pd \in \Sp^2, \pd \notin C$ is the angle between the two great circles that both pass through $\pd$ and are tangent to $C$
\end{definition}

First, we pick a cap $C_j$ , $j \in I$ with the largest spherical radius. There's at least two such caps, any one will do, we will denote it as $C_{max}$, its centre is $\boldsymbol{c}_{max}$, and its radius is $r_{max}$. Then we rotate the cap body so that ${c}_{max} = \ed_3$. Now consider the coordinate great circles $G_1, G_2, G_3$. If all the caps are separated by these circles, then the 6 directions $\lacket \pm \ed_i \racket$ will illuminate the cap body. 

Suppose there is a cap $C_p$ that is not separated by $G_1, G_2, G_3$ and, hence, intersects all of them. In particular, since $G_1$ and $G_2$ pass through $\ed_3$, the view angle $2\alpha$ of $C_p$ from $\ed_3$ is at least $\pi/2$ (see Fig. \ref{fig:symprobcap}).

Let $\boldsymbol{c}_p$ be the centre of the cap $C_p$, and $r_p$ the spherical radius of the cap. We will show that such a cap exists only if $r_{max} = r_p = \pi/4$, and will demonstrate that such a cap configuration still can be separated by three great circles.  Denote the spherical distance between $\boldsymbol{c}_{max}$ and $\boldsymbol{c}_p$ as $d_p$. We assume $d_p \leq \pi/2$, if this is not the case, we will switch to the antipodal cap of $C_p$.

\begin{figure}[h]
    \centering
    \includegraphics[scale=0.45]{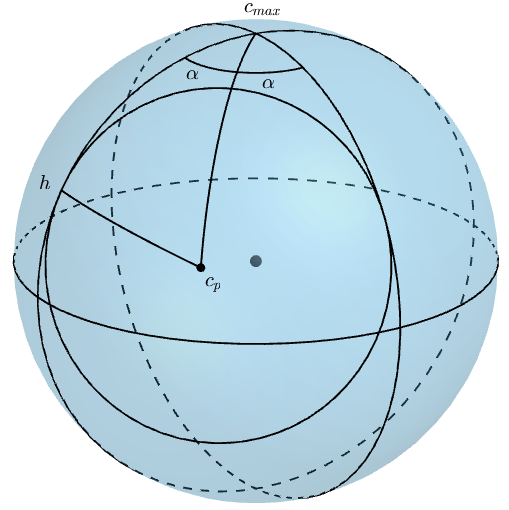}
    \caption{A cap with centre $\boldsymbol{c}_p$ has a view angle $2\alpha \geq \pi/2$ from the point $\boldsymbol{c}_{max}$}
    \label{fig:symprobcap}
\end{figure}

Consider a great circle $F$ that is tangent to $C_p$ at a point $\boldsymbol{h}$ and passes through $\boldsymbol{c}_{max}$. Using the sine theorem for the right spherical triangle $\boldsymbol{c}_{max}\boldsymbol{c}_p\boldsymbol{h}$ and the fact that $d_p \geq r_{max} + r_p$, we get the following inequality:

\begin{equation}
\label{eq:3dtrig}
    \sin \alpha = \frac{\sin r_p}{\sin d_p} \leq \frac{\sin r_p}{\sin (r_{max} + r_p)} = \frac{1}{\sin r_{max} \cot r_p + \cos r_{max}} 
\end{equation}

\begin{itemize}
    \item Case 1: $r_{max} > \pi/4$
    
    The spherical distance between $\boldsymbol{e}_3$ and $-\boldsymbol{e}_3$ is  $\pi \geq 2r_{max} + 2r_p$. Hence $r_p \leq \pi/2 - r_{max}$, which, in this case, leads to $r_p < \pi/4$. Together with (\ref{eq:3dtrig}) we get:

\begin{equation}
    \sin \alpha \leq \frac{1}{\sin r_{max} \cot r_p + \cos r_{max}} < \frac{1}{\sin r_{max}  + \cos r_{max}} < \frac{1}{\sqrt2}
\end{equation}

Which shows that $2\alpha < \pi/2$. Hence the view angle of any other cap from $\boldsymbol{e}_3$ is strictly less than $\pi/2$.
    
    \item Case 2: $r_{max} \leq \pi/4$
    
    Using the inequality (\ref{eq:3dtrig}) and the fact that $r_p \leq r_{max}$ we can obtain the following inequality: 
    
    \begin{equation} \label{eq:s2_pifour}  
        \begin{split}
            &\sin \alpha \leq \frac{1}{\sin r_{max} \cot r_p + \cos r_{max}} \leq \\    
            &\frac{1}{\sin r_{max} \cot r_p + \cos r_{max}} \leq \frac{1}{\sin r_{max} \cot r_{max} + \cos r_{max}} = \\
            &\frac{1}{2 \cos r_{max}} \leq \frac{1}{\sqrt2}
        \end{split}     
    \end{equation}

Equality is only attained if $\cot r_p = \cot r_{max}$ and $\cos r_{max} = \frac{1}{\sqrt2}$ which is equivalent to $r_{max} = r_p = \pi/4$. For the centrally symmetric cap body it is only possible if $\boldsymbol{c}_p$, the centre of $C_p$, is on the great circle $G_3$. To separate all the caps, we pick the great circles $G_1,G_2$ so that $G_2$ passes through $\boldsymbol{c}_p$ and $G_1$ is orthogonal to both $G_2, G_3$.

This configuration separates all the other caps on the sphere. If there is no more caps with a view angle $\pi/2$ from $\boldsymbol{e}_3$, then no other caps can intersect $G_1$ and $G_2$ simultaneously. If there is some other cap $C_q$ with a centre $\boldsymbol{c}_q$ and a spherical radius $r_q$ with a view angle $\pi/2$, then, as shown in (\ref{eq:s2_pifour}), $r_q = \pi/4$ and $\boldsymbol{c}_q \in G_3$. Now on the circle $G_3$ there are four centres (caps $C_p, C_q$ with their antipodes), and no two centres can be closer than $\pi/2$ to each other. That means those four centres are uniformly distributed with distances between adjacent centres being exactly $\pi/2$, as seen in Fig. \ref{fig:symoct}. 

\begin{figure}[h!]
    \centering
    \includegraphics[scale=0.45]{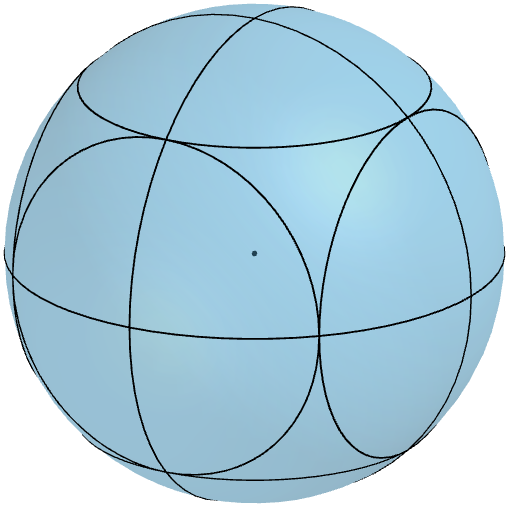}
    \caption{Six caps with radii $\frac{\pi}{4}$}
    \label{fig:symoct}    
\end{figure}

Then the cap $C_p$ and its antipode belong to hemispheres corresponding to circle $G_1$, circle $G_2$ takes care of cap $C_q$ with its antipode, and there are no more caps that intersect all three great circles, as there is no more room on $G_3$. 
\end{itemize}

We have shown that any centrally symmetric packing of spherical caps on $\Sp^2$ can be separated by three greatspheres. That concludes the proof of Theorem \ref{thm:S2}.

\begin{rmk}
This proof is based on solving spherical triangles. In higher dimensions this technique is not as helpful, and our attempts to use the proof for centrally symmetric $\Sp^3$ cap bodies have not been successful so far.
\end{rmk}

\section{Unconditional Cap Bodies and $k$-tangent Caps}\label{sec:ktan}

In this section we will explore possible configurations of the caps on unconditionally symmetrical cap bodies. We do not specify the dimension of a cap body here, everything in this section holds in general $\Ed$ case.

Suppose a coordinate $(d-2)$-greatsphere $G_i$ cuts the cap $C$ off-center, i.e. the centre of $C$ does not lie on $G_i$, but some other interior point $p$ of $C$ does. Since our cap body is symmetric about the hyperplane $H_{e_i,0}$ there is a cap $C'$ $\neq C$ such that $C$ and $C'$ are symmetric about the hyperplane  $H_{\ed_i,0}$. Then $p$ lies both in the $\inter C$ and $\inter C'$, violating the condition that the caps must form a packing (see Fig. \ref{fig:capnopack}).

So if the cap $C$ intersects $G_i$ it  is either  tangent to $G_i$(see Fig. \ref{fig:captan}), or its centre lies on $G_i$ (see Fig. \ref{fig:capcen}).

\begin{figure}[h]
  \begin{subfigure}[b]{0.3\textwidth}
    \includegraphics[width=\textwidth]{{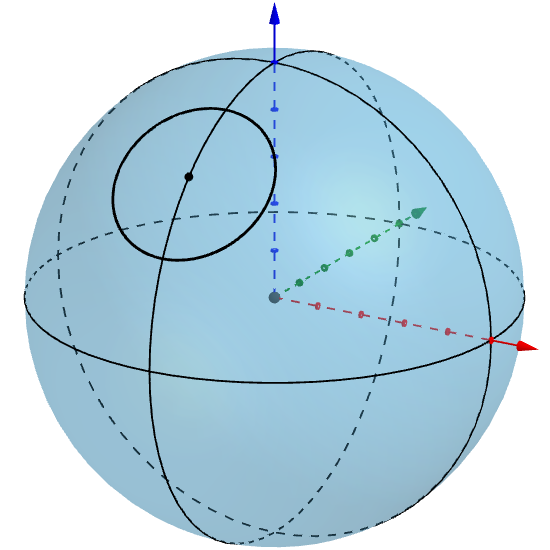}}
    \caption{Centered}
    \label{fig:capcen}
  \end{subfigure}
  \hfill
  \begin{subfigure}[b]{0.3\textwidth}
    \includegraphics[width=\textwidth]{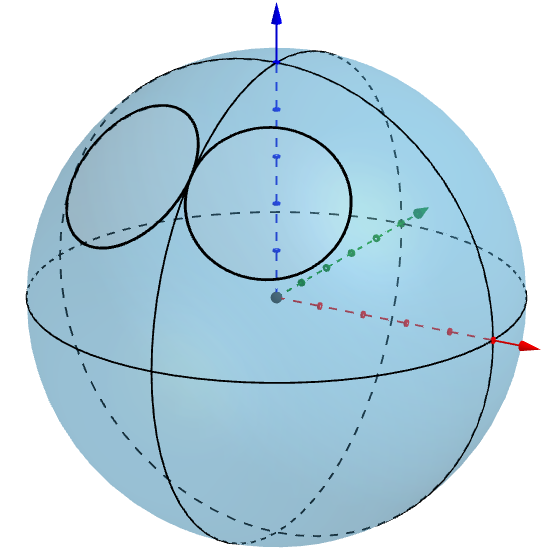}
    \caption{Tangent}
    \label{fig:captan}
  \end{subfigure}
  \hfill
  \begin{subfigure}[b]{0.3\textwidth}
    \includegraphics[width=\textwidth]{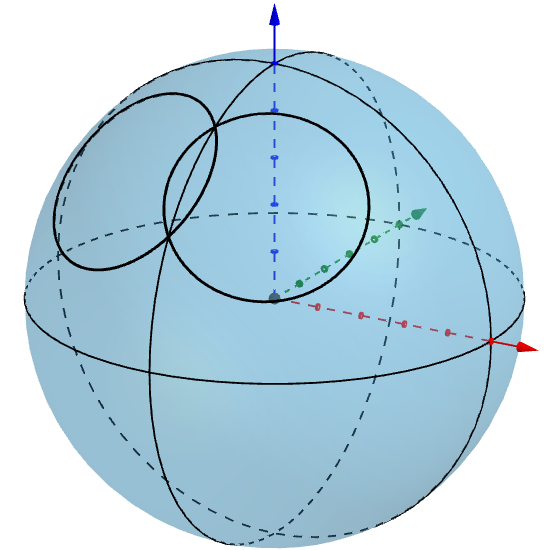}
    \caption{Prohibited, interiors intersect}
    \label{fig:capnopack}
  \end{subfigure}
  \caption{Intersection of a cap and a coordinate greatsphere}
\end{figure}

\indent Again we start with trying to separate all the spherical caps by the coordinate $(d-2)$-greatspheres $G_i$, $i \in \{1, \dots, d\}$. If these greatspheres separate all the caps, then, by Theorem \ref{thm:cbcriterion}, the cap body can be illuminated by $2d$ directions. If, however, there is a spherical cap that is not separated by the system, it must intersect every coordinate greatsphere. For every $G_i$, this cap is either tangent to $G_i$, or has its centre on $G_i$. 

    \begin{definition}
     A \textbf{$k$-tangent cap}, where $k \in \{2, \dots, d\}$ , is a spherical cap $C \subset \Ss$ such that the centre of $C$ lies on the intersection of $d-k$ coordinate greatspheres, and $C$ is tangent to each of the remaining $k$ coordinate greatspheres.
    \end{definition}
    
Note that $k \geq 2$ since a 1-tangent cap is a hemisphere, and the respective cap body is an unbounded cylinder that does not satisfy our definition of the convex body. 

Suppose a $k$-tangent cap $C$ has a spherical radius $r_c$ and its centre lies at a point $\ud_c$ with coordinates $(x_1, \dots, x_d)$. Since our cap body is unconditional, there are also caps congruent to the cap $C$ with coordinates $(\pm x_1, \dots, \pm x_d)$. From this set of caps we can pick the cap with non-negative centre coordinates, so we will assume $x_i \geq 0$.  If the cap's centre is on the greatsphere $G_i$, then $x_i=0$. If, however, the cap is tangential to a greatsphere $G_i$,  its spherical radius $r_c$  is equal to $\pi/2 - \angle(\ud_c, \ed_i)$, the angle between $\ud_c$ and the hyperplane $H_{\ed_i,0}$. Hence, $\sin r_c = \cos (\angle(\ud_c, \ed_i)) = x_i$. 

Without loss of generality suppose that the cap in question is tangential to greatspheres $G_1 \dots G_k$, and its centre lies on greatspheres $G_{k+1} \dots G_d$. Hence, coordinates of the cap centre $\ud_c$ would be $(\underbrace{\sin r_c, \dots, \sin r_c}_k , \underbrace{0, \dots, 0 }_{d-k})$. Since $\sum_{i=1}^{d} x_i^2 = 1$, it follows that $r_c = \arcsin \frac{1}{\sqrt{k}}$.

Suppose there are at least two distinct sets of $k$-tangent caps on a sphere: a set of $k_1$-tangent caps with spherical radii $r_1 = \arcsin \frac{1}{\sqrt{k_1}}$ and a set of $k_2$-tangent caps with radii $r_2 = \arcsin \frac{1}{\sqrt{q}}$, where $p,q \in \{2,3,\dots, d\}$. Let $C_1, C_2$ be the representatives of sets of, respectively, $k_1$-tangent and $k_2$-tangent caps. We pick them so their centres $\boldsymbol{c}_{1}=(x_1, \dots, x_d)$ and $\boldsymbol{c}_2 = (y_1, \dots, y_d)$ have non-negative coordinates. For the two sets of caps to form a packing, the interiors of $C_1$ and $C_2$ must not intersect. 

\begin{lemma}
\label{lem:2setscomp} If a set of $k_1$-tangent caps forms a packing with a set of $k_2$-tangent caps on a unit sphere $\Ss \subset \Ed$, the following inequality holds: 
\begin{equation}
    k_1+k_2-d \leq \sqrt{(k_1-1)(k_2-1)} - 1
\end{equation}
\end{lemma}

\begin{proof}If the interiors of $C_1$ and $C_2$ do not intersect, then the spherical distance between $\boldsymbol{c}_{1}$ and $\boldsymbol{c}_{2}$ is at least the sum of the caps spherical radii $r_1 + r_2$. This condition is equivalent to $\cos \langle \boldsymbol{c}_1, \boldsymbol{c}_2 \rangle \leq \cos(r_1 + r_2)$, as both $\angle (\boldsymbol{c}_1, \boldsymbol{c}_2) $ and $ r_1 + r_2$ are less than $\pi$. Now,  $\cos \langle \boldsymbol{c}_1, \boldsymbol{c}_2 \rangle = \sum_{i=1}^{d} x_i y_i = m \frac{1}{\sqrt{k_1k_2}}$ where $m$ is how many indices $i \in 1 \dots d$ have both $x_i \neq 0$ and $y_i \neq 0$. Since there is exactly $k_1$ non-zero $x$'s and $k_2$ non-zero $y$'s, it follows that $m \geq k_1+k_2-d $. Hence,

\begin{equation}
    \begin{split}
        \frac{k_1+k_2-d}{\sqrt{k_1k_2}} \leq m \frac{1}{\sqrt{k_1k_2}} \leq \cos(r_1 + r_2) \\
        = \cos \left( \arcsin\left( \frac{1}{\sqrt{k_1}}\right) + \arcsin\left( \frac{1}{\sqrt{k_2}}\right) \right) = \frac{\sqrt{k_1-1}\sqrt{k_2-1} - 1}{\sqrt{k_1k_2}}
   \end{split}
\end{equation}
\end{proof}

This condition is based on the optimal case of $m=k_1+k_2-d$, so it is necessary, but not sufficient.

\section {Proof of Theorem \ref{thm:S3}}

Using Lemma \ref{lem:2setscomp}, we can classify the unconditional cap bodies in $\Ee^4$ that are not illuminated by the eight directions $\pm \ed_i,  i \in \{1 \dots 4\}$, based on the $k$-tangent cap configuration. For each possible configuration we will demonstrate a system of four 2-greatspheres that separates the $k$-tangent caps and every other cap that forms a packing with the $k$-tangent caps.

Consider an arbitrary greatsphere $G = \Ss \cap H_{\ud, 0}$ and two spherical caps $C_1, C_2$ with centers $\Od_1$, $\Od_2$, and spherical radii $r_1, r_2$ such that $\Od_1, \Od_2, \ud \in \Ss$, and $r_1, r_2 \in (0, \pi/2)$. Cap $C_1$ is not separated by a greatsphere $G$ if and only if the spherical distance between $\ud$ and $\Od_1$ lies in the range $\left[ \pi/2 - r_1, \pi/2 + r_1 \right]$. This condition can be rewritten as 

\begin{equation}
\label{eq:capgsp}
\langle \ud, \Od_1 \rangle \in \left[-\sin r_1, \sin r_1 \right]
\end{equation}

Caps $C_1, C_2$ form a packing if and only if the spherical distance between their centers $O_1, O_2$ is no less than the sum of the radii $r_1 + r_2$. This condition is equivalent to

\begin{equation}\label{eq:capcap}
    \langle \Od_1, \Od_2 \rangle \leq \cos(r_1 + r_2)
\end{equation}

Consider two sets of $k_1$-tangent and $k_2$-tangent caps forming a packing on $\St \subset \mathbb{E}^4$.

Using Lemma \ref{lem:2setscomp} and cross-checking all 6 pairs of $k_1, k_2 \in \lacket 2,3,4 \racket$ shows that there is only one case with multiple $k$-tangent cap sets to consider, 2 sets of 2-tangent caps such that each coordinate greatsphere is tangential to caps from only one set. Thus we have a total of 4 cases to consider.

\subsection{Eight 2-tangent caps}

Without loss of generality let our two sets of 2-tangent caps have centre coordinates $\left(\pm \frac{1}{\sqrt2}, \pm \frac{1}{\sqrt2}, 0, 0 \right)$ and $\left(0,0,\pm \frac{1}{\sqrt2}, \pm \frac{1}{\sqrt2} \right)$. We will show that an arbitrary spherical cap packing that contains these 2-tangent caps, can be separated by the four greatspheres with the centers  $\left(\frac{1}{\sqrt2}, \pm \frac{1}{\sqrt2}, 0, 0 \right)$ and $\left(0,0, \frac{1}{\sqrt2}, \pm \frac{1}{\sqrt2} \right)$. 
These greatspheres are pairwise orthogonal, so any cap that is not separated by them, has to have a radius that is no less than $\arcsin 1/\sqrt4 = \pi/6$, the inradius of the spherical orthant on $\St$

We are looking for ``stranded'' points, the points with maximum spherical distance to the nearest $k$-tangent cap. If this distance is less than $\pi/6$, then we cannot fit a cap of a radius $\pi/6$ on a sphere around the initial $k$-tangent cap construction, and hence, no other cap can intersect all 4 greatspheres. This technique is helpful with the first two cases.

Consider a countable packing of caps $\left\{C_i \mid i \in I \right \}$ on the $\Ss$ sphere. For every point  $\pd \in \Ss$ and each cap $C_j, j \in I$ there is a non-negative spherical distance $d(\pd, C_j)$ between the point and each cap. For each cap $C_j$ we can construct its spherical Voronoi cell $V_j$ (see Fig. \ref{voronoi}) , a closed set of all the points on the sphere for which $C_j$ is the nearest cap, or one of the nearest caps:

$$V_j = \left\{ \xd \in \Ss \mid d(\xd, C_j) \leq d(\xd, C_i) \mbox{ for any } i \in I \right\} $$

\begin{figure}[h]
\centering
\includegraphics[width=5cm, keepaspectratio]{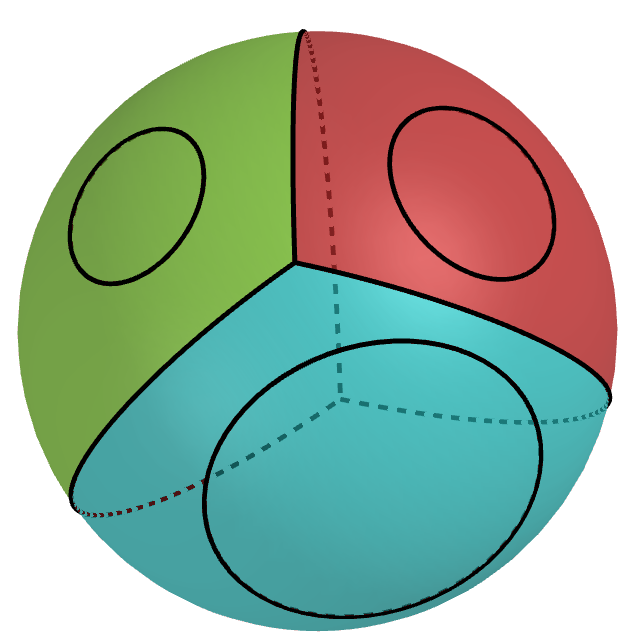}
\caption{Voronoi cells of three caps on $\mathbb{S}^2$}
\label{voronoi}
\end{figure}

All eight 2-tangent caps are images of one cap under the symmetry group that consists of reflections about coordinate hyperplanes and the mirror symmetry that maps $(x_1, x_2, x_3, x_4)$ to $(x_3, x_4, x_1, x_2)$, so all the Voronoi cells are congruent and are the same as the Voronoi cells of the cap centers. So any ``stranded'' point would have to be on a boundary of a cell, otherwise its distance to the nearest cap could be increased by moving the point slightly towards the boundary.

Since the Voronoi cells of all eight 2-tangent caps are congruent, we can consider arbitrary cap, like the one with the centre $\left(\frac{1}{\sqrt2}, \frac{1}{\sqrt2}, 0, 0 \right)$. A point $\boldsymbol{p} = (x_1, x_2, x_3, x_4)$  in its Voronoi cell is closer to $\left(\frac{1}{\sqrt2}, \frac{1}{\sqrt2}, 0, 0 \right)$ than to any of the points $\left(\pm \frac{1}{\sqrt2}, \pm \frac{1}{\sqrt2}, 0, 0 \right)$, or any of the points $\left(0, 0, \pm \frac{1}{\sqrt2}, \pm \frac{1}{\sqrt2} \right)$. It also, obviously, lies on $\Sp^2$. Writing down these conditions via inner product yields the following system:
\begin{equation}
    \begin{cases} \label{2t_x8_voron} x_1 + x_2 \geq \pm x_3 \pm x_4 \\ x_1 + x_2 \geq \pm x_1 \pm x_2 \\ \sum_{j=1}^{4} x_j^2 = 1 \end{cases} 
\end{equation}
 
Keeping these conditions in mind we want to maximize the angle between $\left(\frac{1}{\sqrt2}, \frac{1}{\sqrt2}, 0, 0 \right)$ and $(x_1, x_2, x_3, x_4)$, i.e., minimize the sum $x_1 + x_2$. Equation $x_1 + x_2 \geq \pm x_1 \pm x_2$ yields $x_1, x_2 \geq 0$.


Suppose $x_1+x_2 = A$, then $2A^2 \geq (x_1+x_2)^2 + (|x_3| + |x_4|)^2 = 1 + 2x_1x_2 + 2 |x_3x_4| \geq 1$ with equality attained, for example, with the point $\left(\frac{1}{\sqrt2}, 0 \frac{1}{\sqrt2}, 0 \right)$. Spherical distance between this point and the 2-tangent cap with the centre in $\left(\frac{1}{\sqrt2}, \frac{1}{\sqrt2}, 0, 0 \right)$  equals to $\pi/12$ which is less than $\pi/6$. So any other cap can not intersect all four greatspheres. 

\subsection{Sixteen 4-tangent caps}

In this case, the 4-tangent caps have centre coordinates $(\pm \frac12, \pm \frac12, \pm \frac12, \pm \frac12)$ and spherical radii $r = \arccos{1/\sqrt4} = \pi/6$. We will show that this configuration can be separated by the following four (d-2)-greatspheres: 
$$F_{1,2} = \left\{\xd \in \St \mid  \left\langle \xd,\left( \frac{1}{\sqrt2},  \pm \frac{1}{\sqrt2}, 0, 0 \right) \right\rangle = 0 \right\} $$ 
$$F_{3,4} = \left\{\xd \in \St \mid  \left\langle \xd,\left( 0,0, \frac{1}{\sqrt2},  \pm \frac{1}{\sqrt2} \right) \right\rangle = 0 \right\} $$

In this case, and the further cases, one can verify that the greatspheres separate the 4-tangent caps using the equation (\ref{eq:capgsp}). Next, we need to investigate all the possible caps with radius at least $\pi/6$ that would form a packing on the sphere. Once again, we can pick any cap, like the cap $C_1$ with centre at $O_1 = (\frac12,  \frac12,  \frac12, \frac12)$, and then the other 15 4-tangent caps are images of this cap under the group generated by the reflections about all the coordinate planes. Hence the Voronoi cells of all the 4-tangent caps are congruent and are the same as Voronoi cells of caps centres.

Here are the equations that characterize the Voronoi cell of $C_1$: 
\begin{equation}
    \begin{cases} x_1^2 + x_2^2 + x_3^2 + x_4^2 = 1 \\ x_1 + x_2 + x_3 + x_4 \geq \pm x_1 \pm x_2 \pm x_3 \pm x_4  \end{cases}
\end{equation}

These conditions describe the orthant with non-negative coordinates. Within this closed orthant we need to find the point that is farthest away from $(\frac12,  \frac12,  \frac12, \frac12)$, i.e. a point with minimum $x_1 + x_2 +x_3 +x_4$. 
\begin{equation}
    (x_1 + x_2 + x_3 + x_4)^2 = 1 + 2\sum_{i>j}x_ix_j \geq 1
\end{equation}

So $x_1 + x_2 + x_3 + x_4 \geq 1$ with equality being achieved on the points $\ed_i$, where $i \in 1 \dots 4$. These points are each $\pi/3$ away from $O_1$, so the distance between $C_1$ and $\ed_i$ is $\pi/6$. No other point in the orthant satisfies this condition, since all the other points have at least two positive coordinates.
The radius of any cap with its centre not in $\pm \ed_i$ would be strictly less than $\pi/6$, and it will not intersect all four greatspheres $F_1, \dots, F_4$.

 Caps with centers $(\pm 1, 0,0,0)$  and $(0, \pm 1,0,0)$ are separated by $F_1$, and caps with centers $(0,0,\pm 1,0)$ and $(0, 0,0, \pm 1)$ are separated by $F_3$. Hence we have a system of 4 greatspheres with mutually orthogonal normal vectors that separate all the possible caps of radius at least $\pi/6$, and there is no larger cap. All the caps with radii less than $\pi/6$ cannot intersect all four mutually orthogonal greatspheres simultaneously, and will also be separated.

\subsection{Four 2-tangent caps}

Here our cap configuration has four 2-tangent caps, $C_1, \dots, C_4$ with radii $\pi/4$. Without loss of generality, cap centre coordinates are $\left( \pm \frac{1}{\sqrt2}, \pm  \frac{1}{\sqrt2}, 0, 0 \right)$. To separate these caps we will use the greatspheres $F_{1,2} = \left\{\xd \in \St \mid  \left\langle \xd,\left( \frac{1}{\sqrt2}, \pm \frac{1}{\sqrt2}, 0, 0 \right) \right\rangle = 0 \right\} $, and the coordinate greatspheres $G_3, G_4$.

Greatspheres $F_1$ and $F_2$ separate the 2-tangent caps. Now we just have to make sure that all the other caps that would still form a packing with these caps, are separated by the greatspheres $\left\{F_1, F_2, G_3, G_4 \right\}$.

Suppose some cap $C_0$ has non-zero intersection with every greatsphere $F_1, F_2, G_3, G_4$. Let this cap have the spherical radius $r_0$, and centre coordinates $(x_1, x_2, x_3, x_4)$. We choose $C_0$ so that its centre coordinates are non-negative. Cap $C_0$ is not $k$-tangent, yet it has non-zero intersection with $G_3$ and $G_4$. So it has to have no intersection with either $G_1$ or $G_2$, without loss of generality, suppose it is $G_1$. Equation (\ref{eq:capgsp}) then yields $x_1 \geq \sin r_0$.

We will also use the fact that $C_0$ forms a packing with the 2-tangent caps $C_1 \dots C_4$. Writing down equation (\ref{eq:capcap}) leads to

\begin{equation} \label{eq:xx}
    \pm \frac{x_1}{\sqrt2} \pm \frac{x_2}{\sqrt2} \leq \cos \left(r_0 + \frac{\pi}{4} \right)
\end{equation}

Taking the non-negative signs at $x_1, x_2$ and simplifying the inequality (\ref{eq:xx}),  we get $x_1 + x_2 \leq \cos r_0 - \sin r_0$. This, together with the assumption that $x_1 > \sin r_0$ we can state that 
\begin{equation}
    0 \leq x_2 < \cos r_0 - 2 \sin r_0 
\end{equation}
Since $\cos r_0 - 2 \sin r_0$ monotonously decreases on the allowed range for $r_0$, it follows that $r_0 < \theta$ where $\cos \theta - 2 \sin \theta = 0$. Then $r_0 < \theta = \arcsin (1/\sqrt5) < \pi/6$ and the cap of the radius $r_0$ can not intersect all four mutually orthogonal greatspheres. So it has to be separated by at least one of $F_1, F_2, G_3, G_4$. 

\subsection{Eight 3-tangent caps}

For this case,  take a 3-tangent cap system $C_1, \dots, C_8$ with centre coordinates $\left( \pm \frac{1}{\sqrt3}, \pm \frac{1}{\sqrt3}, \pm \frac{1}{\sqrt3}, 0 \right)$ and all their radii equal to $r_p = \arcsin \frac{1}{\sqrt3}$. To separate these caps, we will use the same 2-greatsphere we have used in the previous case: $F_1, F_2, G_3, G_4$. Repeated use of the equation (\ref{eq:capgsp}) shows that that system does separate the 2-tangent caps. $C_1 \dots C_8$. 

Suppose there is a cap $C_0$ with radius $r_0$ and centre coordinates $(x_1, x_2, x_3, x_4)$ that is not separated by our 2-greatsphere system. Again, we pick it so that $x_i \geq 0$. Here are the conditions such a cap has to satisfy: 

\begin{itemize}

    \item It would have to form a packing with the 3-tangent caps. This condition is equivalent to 
\begin{equation} 
    \pm \frac{x_1}{\sqrt3} \pm \frac{x_2}{\sqrt3}\pm \frac{x_3}{\sqrt3} \leq \cos \left(r_0 + r_p \right)
\end{equation}
    Simplify the right part, and take all signs at $x_i$ to be positive for the strongest statement, and we get 
\begin{equation}\label{eq:wasone}
    x_1 + x_2 + x_3 \leq \sqrt2 \cos r_0 - \sin r_0 
\end{equation}    
    \item The cap $C_0$ is not 3-tangent, so at least one of the intersections $C_0 \cap G_1$, $C_0 \cap G_2$ is empty. Without loss of generality, suppose the cap does not intersect $G_1$. Hence 
\begin{equation}\label{eq:wastwo}
    x_1 > \sin r_0,~x_2 = 0 \mbox{ or } x_2 \geq \sin r_0 
\end{equation}
    \item From (\ref{eq:wasone}) and (\ref{eq:wastwo}) we get $0 \leq x_2 + x_3 <  \sqrt2 \cos r_0 - 2\sin r_0$, which yields the following estimate on $r_0$:
\begin{equation}\label{eq:was3}
    \sqrt2 \cos r_0 - 2\sin r_0 > 0 \Rightarrow r_0 < \arcsin 1/\sqrt3 
\end{equation}    
    \item $C_0$ intersects $G_3$ and $G_4$, so 
\begin{equation}\label{eq:was4}
    x_3 = 0 \mbox{ or } x_3 = \sin r_0,~  x_4 = 0 \mbox{ or } x_4 = \sin r_0 
\end{equation}
    \item Cap $C_0$ intersects $F_1$ and $F_2$. Writing down corresponding equation (\ref{eq:capgsp}) results in
\begin{equation}\label{eq:was5}
    x_1 + x_2 \leq \sqrt2 \sin r_0 
\end{equation}
    \end{itemize}
    
    Now we will show that a cap can not satisfy all these conditions simultaneously.
    
    Suppose $x_2 \neq 0$. Then $x_2 \geq \sin r_0$, which, together with (\ref{eq:was5}), yields $x_1 \leq (\sqrt2 -1) \sin r_0 < \sin r_0$, which contradicts (\ref{eq:was3}). So $x_2 = 0$.
    
    Suppose $x_3 \neq 0$. Then $x_3 = \sin r_0$. Then from (\ref{eq:wasone}) and (\ref{eq:wastwo}) we get $\sin r_0 < x_1 \leq \sqrt2 \cos r_0 - 2 \sin r_0$. It follows that $\sqrt2 \cos r_0 - 3 \sin r_0 >0$. Then $r_0 < \arcsin (\sqrt2 / \sqrt{11}) < \pi/6$, and then cap $C_0$ can not intersect four pairwise orthogonal 2-greatspheres. So $x_3 = 0$.
    
   Since the cap centre lies on $\St$, we get $x_1^2 + x_4^2 = 1$, and from (\ref{eq:was3}),(\ref{eq:was5}) we get $x_1 < \sqrt2/\sqrt3$. Hence, $x_4=\sin r_0$, because th only other option is $x_4 = 0$, which means $x_1 = 1$, contradicting $x_1 < \sqrt2/\sqrt3$. So $x_4 = \sin r_0$.     
    But then $x_4^2 = 1 - x_1^2 > 1/3$, which is incompatible with $r_0 < \arcsin 1/\sqrt3$. Hence there can not be a cap $C_0$ that is distinct from a system of 3-tangent caps, forms a packing with those caps, and is not separated by 2-greatspheres $F_1, F_2, G_3, G_4$.
    
    Now, for any possible unconditional packing of spherical caps on $\Sp^3$ we have shown that there are four pairwise orthogonal 2-greatspheres that separates every cap in the packing. That concludes the proof of Theorem $\ref{thm:S3}$. 
    
    \section{Concluding Remarks}
    
    So far we have failed to find the cap body in $\Ee^3$ with an illumination number higher than 6. We suppose, that 6 is, indeed, the upper estimate for the three-dimensional cap body illumination number.
    
    The illumination estimates we have obtained for the cap bodies with symmetry are sharp. However, we have not completely characterized the centrally symmetric cap bodies that require precisely 6 illumination directions in $\Ee^3$, nor the unconditional cap bodies with illumination number 8 in $\Ee^4$. 
    
    \bigskip
    
    On behalf of all authors, the corresponding author states that there is no conflict of interest.

\bibliographystyle{plain}
\bibliography{references} 

\vfill

\noindent Ilya Ivanov 

\noindent \small{Department of Mathematics and Statistics, University of Calgary, Canada}

\noindent \small{E-mail: \texttt{ilya.ivanov1@ucalgary.ca}}

\bigskip

\noindent and

\bigskip

\noindent Cameron Strachan

\noindent \small{Department of Mathematics and Statistics, University of Calgary, Canada}

\noindent \small{E-mail: \texttt{braden.strachan@ucalgary.ca}}

\end{document}